\documentclass[11pt,a4paper]{article}
\usepackage{amsfonts,amsgen,amstext,amsbsy,amsopn,amsfonts,amssymb,amscd}
\usepackage[leqno]{amsmath}
\usepackage[amsmath,amsthm,thmmarks]{ntheorem}
\usepackage{epsf,epsfig}
\usepackage{float}
\usepackage{dsfont}
\usepackage{ebezier,eepic}
\usepackage{color}
\usepackage{tikz}
\usepackage{multirow}
\usepackage{mathrsfs}
\usepackage{graphicx}
\usepackage{subfigure}
\newtheorem{thm}{Theorem}[section]

\newtheorem{prop}[thm]{Proposition}

\newtheorem{lem}[thm]{Lemma}

\newtheorem{false statement}{False statement}
\newtheorem{cor}[thm]{Corollary}
\newtheorem{fact}[thm]{Fact}

\theoremstyle{definition}
\newtheorem{defn}[thm]{Definition}

\newtheorem{conj}[thm]{Conjecture}

\makeatletter \@addtoreset{equation}{section}

\def\hh{\mathcal{H}}

\def\hht{\mathcal{T}}

\def\hf{\mathcal{F}}
\def\hg{\mathcal{G}}

\def\hb{\mathcal{B}}

\def\hr{\mathcal{R}}

\begin{document}
\title{\bf\Large On resilient hypergraphs}
\date{}
\author{Peter Frankl$^1$, Jian Wang$^2$\\[10pt]
$^{1}$R\'{e}nyi Institute, Budapest, Hungary\\[6pt]
$^{2}$Department of Mathematics\\
Taiyuan University of Technology\\
Taiyuan 030024, P. R. China\\[6pt]
E-mail:  $^1$frankl.peter@renyi.hu, $^2$wangjian01@tyut.edu.cn
}
\maketitle

\begin{abstract}
The matching number of a $k$-graph is the maximum number of pairwise disjoint edges in it. The $k$-graph is called $t$-resilient if omitting $t$ vertices never decreases its matching number. The complete $k$-graph on $sk+k-1$ vertices has matching number $s$ and it is easily seen to be $(k-1)$-resilient. We conjecture that this is maximal for $k=3$ and $s$ arbitrary. The main result verifies this conjecture for $s=2$. Then Theorem \ref{thm-main3} provides a considerable improvement on the known upper bounds for $s\geq 3$.
\end{abstract}

\section{Introduction}

Let $[n]=\{1,2,\ldots,n\}$ be the standard $n$-set.  Let $2^{[n]}$ be the family of all subsets of $[n]$ and $\binom{[n]}{k}$  the family of all $k$-subsets of $[n]$. A family $\hf\subset \binom{[n]}{k}$ is called a {\it $k$-uniform hypergraph} or a {\it $k$-graph}. For $\hf\subset 2^{[n]}$, the {\it matching number} $\nu(\hf)$ is defined as the maximum number of pairwise disjoint members in $\hf$.

One of the most important open problems in extremal set theory is the following.
\begin{conj}[Erd\H{o}s Matching Conjecture \cite{E65}]\label{conj-1}
Suppose that $s$ is a positive integer, $n\geq(s+1)k$ and $\hf\subset {[n]\choose k}$ satisfies $\nu(\hf)\leq s$. Then
\begin{align}\label{1-1}
|\hf|\leq\max\left\{{n\choose k}-{n-s\choose k},{(s+1)k-1\choose k}\right\}.
\end{align}
\end{conj}

The two simple constructions showing that \eqref{1-1} is best possible (if true) are
$$\mathcal{E}(n,k,s)=\left\{E\in{[n]\choose k}\colon E\cap [s]\neq \emptyset\right\} \text{ and } {[(s+1)k-1]\choose k}.$$
Erd\H{o}s proved that \eqref{1-1} is true and up to isomorphic $\mathcal{E}(n,k,s)$ is the only optimal family provided that $n>n_0(k,s)$. The bounds for $n_0(k,s)$ were subsequently improved by Bollob\'{a}s, Daykin and Erd\H{o}s \cite{BDE}, Huang, Loh and Sudakov \cite{HLS}. The current best bounds establish \eqref{1-1} for $n>(2s+1)k$ (\cite{F13}) and for $s>s_0$, $n>\frac{5}{3}sk$ (\cite{FK}).

The $k=3$ case of Erd\H{o}s Matching Conjecture was settled by the first author \cite{F12}.

\begin{thm}[\cite{F12}]
Let $\hf\subset \binom{[n]}{3}$ be a family with matching number $s$. If $n\geq 3(s+1)$, then
\[
|\hf|\leq\max\left\{{n\choose 3}-{n-s\choose 3},{3s+2\choose 3}\right\}.
\]
\end{thm}

Let us recall some useful notations: For $A\subset B$ and a family $\hf$, define
\[
\hf(A,B) = \left\{F\setminus A\colon F\in \hf,\ F\cap B=A\right\}.
\]
We also use $\hf(A)=\hf(A,A)$ and $\hf(\bar{A})=\hf(\emptyset,A)$.

\begin{defn}
Let $\hf\subset \binom{[n]}{k}$ be a family. We say that $\hf$ is $t$-resilient if $\nu(\hf(\bar{T}))=\nu(\hf)$ for any $T\subset \binom{[n]}{\leq t}$.
\end{defn}

If $T\in \hht$ is an edge then $\nu(\hht(\bar{T}))<\nu(\hht)$ is obvious. That is, no $k$-graph is $k$-resilient. In this paper we focus
on $(k-1)$-resilient $k$-graphs. The case of ordinary graphs, that is, $k=2$ is well-understood (cf. \cite{LP}). In  particular, for a 1-resilient 2-graph $\hg$ with matching number $s$, $|\hg|\leq \binom{2s+1}{2}$ holds and $K_{2s+1}$ is the only graph attaining equality.

\begin{conj}
Suppose that $\hht$ is a 2-resilient 3-graph with $\nu(\hht)=s$. Then
\begin{align}\label{ineq-1.1}
|\hht| \leq \binom{3s+2}{3}.
\end{align}
\end{conj}

In the case $s=1$ \eqref{ineq-1.1} was proved by Erd\H{o}s and Lov\'{a}sz \cite{EL}. We establish \eqref{ineq-1.1} for $s=2$.

\begin{defn}
Define
\[
m(k,s) =\max\left\{|\hf|\colon \hf \mbox{ is a $(k-1)$-resilient $k$-graph with } \nu(\hf)=s\right\}.
\]
\end{defn}

Our main result is the following.
\begin{thm}\label{thm-main}
\begin{align}\label{ineq-main}
m(3,2) = \binom{8}{3} = 56.
\end{align}
\end{thm}

Being a special case of a conjecture, it might not sound too impressive. However, {\it exact} results on this type of questions are far and few between.

For $u,v\in [n]$, define
\[
\hht_{uv} =\{T\in \hht\colon T\cap \{u,v\}\neq \emptyset\}.
\]

 One of the main lemmas in the proof of \eqref{ineq-main} is the following.

\begin{lem}\label{lem-main}
Let $\hht\subset \binom{[n]}{3}$ be a 2-resilient family with matching number 2. Then for any $u,v\in [n]$,
\[
|\hht_{uv}|\leq 36.
\]
\end{lem}

Note that the family  $\binom{[8]}{3}$ shows that the upper bound $36$ is best possible.

The best known general upper bound on $m(k,s)$ is due to Lov\'{a}sz \cite{lovasz}.

\begin{thm}[\cite{lovasz}]\label{thm-lovasz}
\begin{align}\label{ineq-m3s}
m(k,s)\leq (ks)^s.
\end{align}
\end{thm}

For $k=3$, it gives $m(3,s) \leq 27s^3$. The next result provides an improvement.
\begin{thm}\label{thm-main3}
For $s\geq 3$,
\[
m(3,s) \leq \frac{73}{6}s^3+50.
\]
Moreover, for $s\geq 21$,
\[
m(3,s) \leq \frac{73}{6}s^3.
\]
\end{thm}

Actually, the case $s=1$ has received the most attention. Recall that a $k$-graph $\hf$ with $\nu(\hf)=1$ is called {\it intersecting}. The case $s=1$ of the Erd\H{o}s Matching Conjecture is the Erd\H{o}s-Ko-Rado Theorem, one of the central results in extremal set theory.

For a $k$-graph $\hf$  its {\it covering number} (or {\it transversal number}) $\tau(\hf)$ is defined as
\[
\min\left\{|Y|\colon Y\cap F\neq \emptyset \mbox{ for all } F\in \hf\right\}.
\]
If $\hf$ is intersecting then $\tau(\hf)\leq k$. Then $\hf$ is $(k-1)$-resilient if and only if $\tau(\hf)=k$.

Moreover, if $\tau(\hf)=k$ then one can keep adding new edges ($k$-sets) that intersect all the edges already in $\hf$. Eventually
this process will stop. The end result is a {\it maximal} intersecting $k$-graph $\tilde{\hf}$, i.e., an intersecting $k$-graph with $\tau(\tilde{\hf})=k$
and such that to every $k$-set $H\notin \tilde{\hf}$ there is some $F\in \tilde{\hf}$ with $H\cap F=\emptyset$. The special case $s=1$ of Theorem \ref{thm-lovasz} states
\[
|\tilde{\hf}| \leq m(k,1) \leq k^k,
\]
it is due to Erd\H{o}s and Lov\'{a}sz \cite{EL}. In that paper Erd\H{o}s and Lov\'{a}sz  showed
\begin{align}\label{ineq-1.2}
m(k,1)\geq \lfloor (e-1)k!\rfloor.
\end{align}
Note that already for $k=4$,
\[
\binom{4+3}{4}=35<41=\lfloor (e-1)4!\rfloor.
\]

Lov\'{a}sz \cite{lovasz} conjectured that equality holds in \eqref{ineq-1.2}. This was disproved in \cite{FOT}, where a construction showing
\[
m(k,1)\geq \left(\frac{k}{2}+o(1)\right)^k \mbox{ was exhibited}.
\]
The upper bound $k^k$ was somewhat improved in  \cite{Tuza}, \cite{Cherkashin}, \cite{AT}, \cite{F19}, \cite{Zakharov}. However, the ratio of the best upper and lower bounds is still $2^{k(1+o(1))}$. In particular, for $k=4$ we only know
\[
42\leq m(4,1)\leq 175.
\]

\section{Some simple facts and the proof of Theorem \ref{thm-main3}}

Let $\hht\subset \binom{[n]}{3}$ be a 2-resilient family with matching number $s$, $s\geq 2$. Define
\[
\hr =\hht^s:= \left\{T_1\cup T_2\cup \cdots \cup T_s\colon T_1,T_2,\ldots, T_s\in \hht\mbox{ form a matching}\right\}.
\]

\begin{fact}\label{fact-1}
Suppose that $T_1, T_2,\ldots, T_s$ and $V_1,V_2,\ldots, V_s$ are maximal matchings in $\hht$. Then there exists a permutation $\sigma$ of $\{1,2,\ldots,s\}$ such that $T_i\cap V_{\sigma(i)}\neq \emptyset$, $1\leq i\leq s$.
\end{fact}

\begin{proof}
Make a bipartite graph $\hb$ with partite sets $\{T_1,T_2,\ldots,T_s\}$ and $\{V_1,V_2,\ldots,V_s\}$. Put an edge $(T_i,V_j)$ if and only if $T_i\cap V_j\neq \emptyset$. The fact is equivalent to the existence of a perfect matching in $\hb$. Thus it is sufficient to show that the K\"{o}nig-Hall condition is verified. For an arbitrary $X\subset [s]$ consider the set $Y=\{j\colon \exists i\in X,\ T_i\cap V_j\neq \emptyset\}$. Replacing the edges $V_j$ with $j\in Y$ by the edges $T_i$ with $i\in X$ from $V_1,V_2,\ldots,V_s$ produces a matching. Thus $|Y|\geq |X|$ follows from $\nu(\hht)=s$.
\end{proof}

\begin{cor}
The family $\hr$ is $s$-intersecting. Moreover if $R=T_1\cup  T_2\cup \ldots\cup T_s\in \hr$ and $Q=V_1\cup  V_2\cup \ldots\cup V_s\in \hr$ satisfy $|R\cap Q|=s$ then there is a unique permutation $\sigma$ of $[s]$ such that
\[
|T_i\cap V_j| = \left\{
                \begin{array}{ll}
                  1, & \mbox{ if } j=\sigma(i)\\[2pt]
                  0, & otherwise.
                \end{array}
              \right.
\]
\end{cor}

The distinct sets $S_0,\ldots,S_r$ are said to form a {\it pseudo sunflower} of size $r+1$ and center $C$ if $C\subsetneq S_0$ and the sets $S_i\setminus C$ are pairwise disjoint, $0\leq i\leq r$.

If $S_0,\ldots,S_r$ is a pseudo sunflower with center $C\subsetneq F_0$, then for any $C'$ with $C\subset C'\subsetneq S_0$ the sets $S_i\setminus C'$ are pairwise disjoint. Hence $S_0,S_1,\ldots,S_r$ form a pseudo sunflower with center $C'$ as well. Based on this observation throughout this paper we always assume that the center $C$ of a pseudo sunflower satisfies $|S_0\setminus C|=1$.

\begin{thm}[F\"{u}redi \cite{Fu80}, cf. also \cite{F2022}]\label{psudo-sunflower}
Let $k,r$ be positive integers and let $\hf$ be a $k$-graph not containing any pseudo sunflower of size $r+1$. Then
\begin{align}\label{ineq-pseudosunflower}
|\hf| \leq r^k.
\end{align}
\end{thm}

\begin{fact}\label{fact-2}
$\hht$ does not contain a pseudo sunflower of size $3s+1$.
\end{fact}
\begin{proof}
By symmetry, let $S_0,S_1,\ldots,S_{3s}\in \hht$ form a pseudo sunflower with center $C\subsetneq S_0$. Let $R\in \hr(\overline{C})$ ($R$ exists by 2-resilience). As $|R|=3s$ we can find $j$ with $R\cap (S_j\setminus C)=\emptyset$. It follows that $R\cap S_j=\emptyset$, a contradiction.
\end{proof}

\begin{cor}\label{cor-1}
For any $P\in \binom{[n]}{2}$, $|\hht(P)|\leq 3s$.
\end{cor}

Let us fix a matching $T_1,T_2,\ldots,T_s\in \hht$ and let $R=T_1\cup T_2\cup \cdots \cup T_s$.

\begin{fact}\label{fact-3}
For any $x\in R$,  $\hht(x,R)$ does not contain a pseudo sunflower of size $2s$. Moreover, if $\hr$ is $(s+1)$-intersecting then $\hht(x,R)$ does not contain a pseudo sunflower of size $2s-1$.
\end{fact}

\begin{proof}
Let $x\in R$ and $F_0,F_1,\ldots,F_{p-1}$ a pseudo sunflower with center $y$ in $\hht(x,R)$. Using 2-resilience fix $R'\in \hr(\bar{x},\bar{y})$. By $\nu(\hht)=s$, $(R'\setminus R)\cap F_\ell\neq \emptyset$ for all $0\leq \ell <p$. In particular,
\begin{align}\label{ineq-star}
|R'|-|R\cap R'| \geq p.
\end{align}

Since $|R'|=3s$ and $\hr$ is $s$-intersecting for $p=2s$ we get a contradiction unless $|R\cap R'|=s$. However in this case letting $R=T_1\cup T_2,\ldots \cup T_s$, $R'=V_1\cup V_2,\ldots \cup V_s$  we may assume $|T_i\cap V_i|=1$, $1\leq i\leq s$ and $x\in T_1\setminus V_1$. By $|V_1\setminus T_1|\leq 2<p=2s$ we can choose $0\leq k <p$ so that $F_k\cap V_1=\emptyset$. Then $\{x\}\cup F_k,V_1,T_2,\ldots,T_s$ form a matching of size $s+1$, a contradiction.

Assuming $p=2s-1$ and that $\hr$ is $(s+1)$-intersecting, in view of \eqref{ineq-star} the only remaining case is $|R\cap R'|=s+1$. By $\nu(\hht)=s$ we
can still  assume that $V_i\cap T_i\neq \emptyset$, $1\leq i\leq s$. The above argument works unchanged if $V_1\cap T_j=\emptyset$ for $2\leq j\leq s$. On the other hand, if $|V_1\cap R|\geq 2$ then using $|(V_2\cup \ldots\cup V_s)\setminus R|\leq 2s-2<p=2s-1$ we may fix $F_k$ to satisfy $(F_k\cup \{y\}) \cap V_i=\emptyset$ for $2\leq i\leq s$. Thus $(F_k\cup \{y\})\cup V_2\cup \ldots \cup V_s=:R''\in \hr$. However $|R\cap R''|<|R\cap R'|$ contradicting the $(s+1)$-intersecting property of $\hr$.
\end{proof}

\begin{fact}\label{fact-4}
Let $P\in \binom{R}{2}$. Then $|\hht(P,R)|\leq 2s$. Moreover, if $\hr$ is $(s+1)$-intersecting then  $|\hht(P,R)|\leq 2s-1$.
\end{fact}
\begin{proof}
Assume that $\{x_i\}\in \hht(P,R)$ for $i=1,2,\ldots, 2s+1$. Then $P\cup \{x_1\},P\cup \{x_2\},\ldots,P\cup \{x_{2s+1}\},T_1,T_2,\ldots,T_s$ form a pseudo sunflower with center $P$, contradicting Fact \ref{fact-2}. Thus $|\hht(P,R)|\leq 2s$.

If $\hr$ is $(s+1)$-intersecting and $\{x_i\}\in \hht(P,R)$ for $i=1,2,\ldots, 2s$. Let $R'\in \hr$ satisfy $R'\cap P=\emptyset$. Then $\{x_1,x_2,\ldots,x_{2s}\}\subset R'$ implies $|R'\cap R|=s$, contradicting the $(s+1)$-intersecting property of $\hr$.
\end{proof}

For a graph $\hg$ let $\Delta(\hg)$ denote its maximum degree.

\begin{fact}\label{fact-4.4}
Suppose that a graph $\hg$ contains no pseudo sunflower of size three. Then $\hg$ is a triangle or a subgraph of a $C_4$.
\end{fact}

\begin{proof}
If $\hg$ is intersecting then $\Delta(\hg)\leq 2$ implies $\hg\subset C_3$. Otherwise let $(1,2),(3,4)$ be two disjoint edges in $\hg$. Then $\hg\subset \binom{[4]}{2}$ and via $\Delta(\hg)\leq 2$, $\hg \subset C_4$ follows.
\end{proof}

 For disjoint sets $A,B$ let
\[
A\times B=\{(a,b)\colon a\in A,\ b\in B\}.
\]

\begin{fact}\label{fact-6}
Let $d\geq 3$. Suppose that $Q$ is a set and $\hg_1,\hg_2$ are graphs with $\Delta(\hg_i)\leq d$ and  $|\hg_i|>(|Q|+1)d$. Then there exist $E_i\in \hg_i$ with $E_1\cap E_2=\emptyset$ and $E_i\cap Q=\emptyset$, $i=1,2$.
\end{fact}

\begin{proof}
Define $\hh_i=\{E_i\in \hg_i\colon E_i\cap Q=\emptyset\}$ . Obviously,
\[
|\hh_i|\geq |\hg_i| -|Q|d \geq d+1.
\]
Hence we can find a pair $E_i,F_i$ of disjoint edges in $\hh_i$. If $\hh_1$ and $\hh_2$ are cross-intersecting, then $\hh_i\subset E_{3-i}\times F_{3-i}$, $i=1,2$ follows. Using $d+1\geq 4$ we infer that $E_1\cup F_1=E_2\cup F_2$ is the same 4-set. Now $|\hh_1|+|\hh_2|\leq \binom{4}{2}$ and $\min|\hh_i|\leq 3$ follow, a contradiction.
\end{proof}

\begin{fact}\label{fact-general1}
Suppose that $\hg_1,\hg_2,\hg_3$ are non-empty graphs of maximum degree at most $d$ ($d\geq 3$). If $|\hg_i|>3d$ then there is a cross-matching $E_1\in \hg_1$, $E_2\in \hg_2$ and $E_3\in \hg_3$.
\end{fact}

\begin{proof}
Fix $E_1\in \hg_1$. By Fact \ref{fact-6} there exist $E_2\in \hg_2$ and $E_3\in \hg_3$ with $E_2\cap E_3=\emptyset$ and $E_i\cap E_1=\emptyset$, $i=2,3$.
\end{proof}

Let us turn to the proof of Theorem \ref{thm-main3}. Let $\hht\subset \binom{[n]}{3}$ be a 2-resilient 3-graph with $\nu(\hht)=s$.
Let  $T_1,T_2,\ldots,T_s\in \hht$ be a matching and let $R=T_1\cup T_2\cup \cdots \cup T_s$. Define
\[
R_0=\left\{x\in R\colon |\hht(x,R)|\leq 6s\right\} \mbox{ and } R_1=R\setminus R_0.
\]

\begin{cor}
$|R_1\cap T_i|\leq 1$.
\end{cor}

\begin{proof}
Suppose that $x,y\in R_1\cap T_i$. Then by Fact \ref{fact-general1} there exist disjoint edges  $E_1\in \hht(x,R)$ and $E_2\in \hht(y,R)$.
It follows that $T_1,\ldots,T_{i-1},T_{i+1},\ldots,T_s, E_1\cup \{x\}, E_2\cup \{y\}$ form a matching, contradicting $\nu(\hht)=s$.
\end{proof}

Without loss of generality, assume that $R_1=\{x_1,x_2,\ldots,x_t\}$ for some $0\leq t\leq s$ and $x_i\in T_i$ for $i=1,2,\ldots,t$.

\begin{fact}\label{fact-general2}
Let $1\leq i<j\leq t$. Suppose that $y_i\in T_i\setminus \{x_i\}$ and $y_j\in T_j\setminus \{x_j\}$. Then $\hht(\{y_i,y_j\},R)=\emptyset$ and $\hht(y_i,R)=\emptyset$.
\end{fact}

\begin{proof}
Assume for contradiction $\{y_i,y_j,z\}\in\hht$ with $z\notin R$. Then apply Fact \ref{fact-6} with $Q=\{z\}$, to find disjoint edges $E_i\in \hht(x_i,R)$ and $E_j\in \hht(x_j,R)$. Now $E_i\cup \{x_i\}, E_j\cup  \{x_j\}$, $\{y_i,y_j,z\}$ form a matching, contradicting $\nu(\hht)=s$.

If $\hht(y_i,R)\neq \emptyset$, let $E\in \hht(y_i,R)$. Then  by $x_i\in R_1$ there exists $E'\in \hht(x_i,R)$ with $E\cap E'=\emptyset$, contradicting $\nu(\hht)=s$. Thus $\hht(y_i,R)=\emptyset$.
\end{proof}

By Facts \ref{fact-general2} and \ref{fact-4}, we infer that for $1\leq i<j\leq t$,
\begin{align}\label{ineq-2.6}
\sum_{P\in T_i\times T_j} |\hht(P,R)| \leq 2s\times 5=10s.
\end{align}

\begin{fact}\label{fact-general3}
For $1\leq i<j\leq s$ and $t+1\leq j\leq s$,
\begin{align}\label{ineq-2.7}
\sum_{P\in T_i\times T_j} |\hht(P,R)| \leq 12s.
\end{align}
\end{fact}
\begin{proof}
Let  $P_1,P_2,P_3$ be a matching in $T_i\times T_j$. By $\nu(\hht)\leq s$, there are no $u,v,w\notin R$ such that $P_1\cup \{u\}, P_2\cup \{v\},P_3\cup \{w\}\in \hht$. It follows that either one of $\hht(P_1,R), \hht(P_2,R), \hht(P_3,R)$ is empty or two of them are an identical 1-element set or the union of them has size two. Since $\hht(P_i,R)\leq 2s$ from Fact \ref{fact-4} and $s\geq 2$, we infer that
\[
|\hht(P_1,R)|+|\hht(P_2,R)|+|\hht(P_3,R)|\leq 4s.
\]
Since $T_i\times T_j$ can be partitioned into 3 groups of matchings, the fact follows.
\end{proof}

Define
\begin{align*}
&\hht_1=\left\{T\in \hht\colon |T\cap R|= 1\right\},\\[3pt]
 &\hht_{21}=\left\{T\in \hht\colon \mbox{ there exists }1\leq i\leq s, |T\cap T_i|\geq 2 \right\},\\[3pt]
 & \hht_{22}=\left\{T\in \hht\colon |T\cap R|= 2,\mbox{ there exist }1\leq i<j\leq s, |T\cap T_i|= 1 =|T\cap T_j|\right\},\\[3pt]
 &\hht_{3}=\left\{T\in \hht\colon \mbox{ there exist }1\leq i<j<k\leq s, |T\cap T_i|= 1 =|T\cap T_j|=|T\cap T_k|\right\}.
\end{align*}
By \eqref{ineq-2.6} and \eqref{ineq-2.7},
\begin{align}\label{ineq-general1}
|\hht_{22}|=\sum_{1\leq i<j\leq s}\sum_{P\in T_i\times T_j} |\hht(P,R)| &\leq \binom{t}{2} 10s +t(s-t)12s+\binom{s-t}{2}12s\nonumber\\[3pt]
&=5st^2+12st(s-t)+6s(s-t)^2-6s^2+st.
\end{align}
By Facts \ref{fact-3} and \ref{fact-general2},
\begin{align}\label{ineq-general2}
|\hht_1|=\sum_{x\in R} |\hht(x,R)|\leq t(2s-1)^2+(s-t)\times 3\times 6s=4s^2 t+18s(s-t)-4st+t.
\end{align}
and
\begin{align}\label{ineq-general3}
|\hht_{21}|=\sum_{1\leq i\leq s}\sum_{P\in \binom{T_i}{2}}|\hht(P)| \leq 3s\times3s=9s^2.
\end{align}

For $1\leq i<j<k\leq s$, define
\[
\hht_{ijk} = \left\{T\in \hht\colon |T\cap T_i|=|T\cap T_j|=|T\cap T_k|=1\right\}.
\]

\begin{fact}\label{fact-general4}
If $1\leq i<j<k\leq t$, then $|\hht_{ijk}|\leq 27-8$. If $1\leq i<j\leq t<k\leq s$, then $|\hht_{ijk}|\leq 27-6$.
\end{fact}

\begin{proof}
If $1\leq i<j<k\leq t$, then by Fact \ref{fact-general1} there exist disjoint sets $P_i\in \hht(x_i,R)$, $P_j\in \hht(x_j,R)$  and $P_k\in \hht(x_k,R)$. It implies that $\hht\cap \binom{T_1\cup T_j\cup T_k \setminus \{x_i,x_j,x_k\}}{3}=\emptyset$. Thus $|\hht_{ijk}|\leq 19$.

If $1\leq i<j\leq t<k\leq s$, then by Fact \ref{fact-general1} there exist disjoint sets  $P_i\in \hht(x_i,R)$ and  $P_j\in \hht(x_j,R)$. It follows that  $\hht\cap \binom{T_1\cup T_j\cup T_k \setminus \{x_i,x_j\}}{3}$ is intersecting. Thus at least 6 edges in $\binom{T_1\cup T_j\cup T_k \setminus \{x_i,x_j\}}{3}$ are missing. Therefore $|\hht_{ijk}|\leq 21$.
\end{proof}

By Fact \ref{fact-general4},
\begin{align}\label{ineq-general4}
|\hht_{3}|=\sum_{1\leq i<j<k\leq s} |T_{ijk}| &\leq 27\binom{s}{3}-8\binom{t}{3}-6\binom{t}{2}(s-t)\nonumber\\[3pt]
&=\frac{9}{2}s^3 -3t^2s+\frac{5}{3}t^3 +t^2+3 s t-\frac{27}{2}s^2+9s-\frac{8}{3}t.
\end{align}
Adding \eqref{ineq-general1}, \eqref{ineq-general2}, \eqref{ineq-general3} and \eqref{ineq-general4}, we conclude that
\begin{align*}
|\hht|\leq 5st^2&+12st(s-t)+6s(s-t)^2+4s^2 t+\frac{9}{2}s^3 -3t^2s+\frac{5}{3}t^3\\[3pt]
     &+\frac{15}{2}s^2-18 s t + t^2+9 s-\frac{5}{3}t=:f(t).
\end{align*}
Checking the second derivative $f''(t)$, it follows that $f'(t)$ is minimal for $t=\frac{4s-1}{5}$. For $s\geq 21$,
\[
 f'(\frac{4s-1}{5})=\frac{2}{15} (6 s^2- 123 s -14 )>0.
\]
Consequently,
\[
|\hht| \leq f(s) = \frac{73}{6}s^3-\frac{19}{2}s^2+\frac{22}{3}s <\frac{73}{6}s^3.
\]
For $3\leq s\leq 20$ and $0\leq t\leq s$, one can check by computation that
\[
|\hht| \leq f(t) <\frac{73}{6}s^3+50.
\]

\section{Two special cases of Theorem \ref{thm-main}}

Let $\hht\subset \binom{[n]}{3}$ be a 2-resilient family with matching number 2 and let $\hr=\hht^2$. Let $A,B\in \hht$ be disjoint and let $R=A\cup B$. Define
\[
\hht_0=\{T\in \hht\colon T\cap A\neq \emptyset \neq T\cap B\}.
\]
\begin{lem}
\[
|\hht_0|\leq 36.
\]
\end{lem}

\begin{proof}
For each $P\in A\times B$ and $P\subset T\in \hht_0$, let $\omega(P,T)=1$ if $|T\cap R|=2$ and  let $\omega(P,T)=\frac{1}{2}$ if $T\subset R$. Define
\[
\omega(P) =\sum_{T\colon P\subset T\in \hht_0} \omega(P,T).
\]
It is easy to check that
\[
|\hht_0| =\sum_{P\in A\times B} \omega(P).
\]
Let $P_1,P_2,P_3$ form a 3-matching of $A\times B$. By Corollary \ref{cor-1} and Fact \ref{fact-4}, we have $|\hht(P_i)|\leq 6$ and $|\hht(P_i,R)|\leq 4$.
Thus,
\begin{align*}
\omega(P_i)= |\hht(P_i,R)|+\frac{1}{2}|\{T\subset R\colon P_i\subset T\in \hht\}|  \leq 4+\frac{2}{2}=5.
\end{align*}

Since there do not exist distinct vertices $\{x\}\in \hht(P_1,R), \{y\}\in \hht(P_2,R), \{z\}\in \hht(P_3,R)$, we infer that either one of them is empty or two of them are an identical 1-element set or the union  of them has size two.  If one of them is empty, then
\begin{align*}
\omega(P_1)+\omega(P_2)+\omega(P_3) \leq 2\times 5+\frac{4}{2}= 12.
\end{align*}
If two of them are an identical 1-element set, then
\begin{align*}
\omega(P_1)+\omega(P_2)+\omega(P_3) \leq 5+2\times 1+2\times \frac{4}{2}= 11.
\end{align*}
If the union  of them has size two, then
\begin{align*}
\omega(P_1)+\omega(P_2)+\omega(P_3) \leq 3\times (2+ \frac{4}{2})= 12.
\end{align*}

Since $A\times B$ can be partitioned into 3 groups of 3-matchings,
\[
|\hht_0| =\sum_{P\in \binom{[6]}{2}} \omega(P) \leq 3\times 12=36.
\]
\end{proof}

Note that even though there is some formal similarity the above statement is completely different from Lemma \ref{lem-main}.

\begin{cor}
If $\hht\subset\binom{[n]}{3}$ satisfies $\nu(\hht)=2$ and $\tau(\hht)=6$, then
\[
|\hht|\leq 56.
\]
\end{cor}

\begin{proof}
Let $A$ and $B$ be two disjoint edges in $\hht$. Then $\hht=\hht_0\cup \hht(\bar{A})\cup \hht(\bar{B})$. By $\nu(\hht)=2$ both $\hht(\bar{A})$ and $\hht(\bar{B})$ are intersecting. By $\tau(\hht)=6$ they both have transversal number 3. That is, they are 2-resilient. By the case $s=1$ of  \eqref{ineq-1.1}, $|\hht(\bar{A})|+|\hht(\bar{B})|\leq 10+10=20$. Thus $|\hht|\leq 10+10+36=56$.
\end{proof}

\begin{prop}
Let $\hht\subset \binom{[n]}{3}$ be a 2-resilient family with matching number 2. If there exist disjoint sets $A,B\in \hht$  such that both $\hht(\bar{A})$ and $\hht(\bar{B})$ are stars, then
\[
|\hht| \leq 56.
\]
\end{prop}

\begin{proof}
Let $A=(x,y,z)$, $B=(u,v,w)$ and $R=A\cup B$. Assume that  $\hht(\bar{A})$ is a star with center $u$ and  $\hht(\bar{B})$ is a star with center $x$.
By Lemma \ref{lem-main}, we infer that $|\hht_{xu}|\leq 36$. We are left to show that $|\hht(\bar{x},\bar{u})|\leq 20$.

Let $\hht_0=\{T\in \hht\colon T\cap \{x,u\}=\emptyset\}$. Since $T\cap (y,z)\neq \emptyset\neq T\cap (v,w)$ for every $T\in \hht_0$, we infer that
\[
|\hht(\bar{x},\bar{u})|= |\hht_0| =\left|\binom{\{y,z,v,w\}}{3}\right|+\sum_{P\in (y,z)\times (v,w) } |\hht_0(P, R)|.
\]
By Fact \ref{fact-4}, $|\hht_0(P, A\cup B)|\leq 4$. Then
\[
|\hht_0|\leq 4+4\times 4=20
\]
 and the proposition follows.
\end{proof}

\section{The proof of Theorem \ref{thm-main} in the case $\hr=\hht^2$ is  3-intersecting}

Throughout this section we assume that $\hht$ is a 2-resilient 3-graph with $\nu(\hht)=2$, $[6]\in \hr=\hht^2$ and  $\hr$ is 3-intersecting. Let us partition $\hht$ into three parts
\[
\hht_i = \left\{T\in \hht\colon |T\cap [6]|=i\right\}, i=1,2,3.
\]

\begin{fact}\label{fact-4.3}
Suppose that $\hr$ is 3-intersecting and $R\in \hr$. If $x,y\in R$ then $\hht(x,R)$ and $\hht(y,R)$ are cross-intersecting.
\end{fact}

\begin{proof}
It is a direct consequence of the 3-intersecting property of $\hr$.
\end{proof}

Let $\hf_1,\hf_2,\ldots,\hf_s\subset 2^{[n]}$. We say that $\hf_1,\hf_2,\ldots,\hf_s$ are {\it overlapping} if there do not exist pairwise disjoint sets $F_1\in \hf_1, F_2\in\hf_2,\ldots,F_s\in \hf_s$.

Note that the integer 6 has three partitions into three positive integers, namely $6=2+2+2$, $6=1+2+3$ and $6=1+1+4$.

\begin{fact}\label{fact-4.1}
Suppose that the non-empty sets $B_1,B_2,B_3$ ($|B_1|\leq |B_2|\leq |B_3|$) partition $[6]$, then (i) and (ii) hold.
\begin{itemize}
  \item[(i)] If $|B_3|\geq 3$ and $\binom{B_3}{3}\cap \hht\neq \emptyset$, then $\hht(B_1,[6])$ and $\hht(B_2,[6])$ are cross-intersecting.
  \item[(ii)] If $|B_3|\leq 3$ then $\hht(B_1,[6])$, $\hht(B_2,[6])$ and  $\hht(B_3,[6])$ are overlapping.
\end{itemize}
\end{fact}

\begin{proof}
 Both statements are immediate consequences of $\nu(\hht)=2$.
\end{proof}

\begin{cor}\label{cor-4.2}
If in Fact \ref{fact-4.1} $|B_1|=|B_2|=|B_3|=2$ and $|\hht(B_1,[6])|\leq |\hht(B_2,[6])|\leq |\hht(B_3,[6])|$, then either $\hht(B_1,[6])=\emptyset$, or $|\hht(B_1,[6])\cup \hht(B_2,[6])|= 1$ or $|\hht(B_1,[6])\cup \hht(B_2,[6])\cup \hht(B_3,[6])|= 2$.
\end{cor}
\begin{proof}
By Fact \ref{fact-4.1} (ii), the statement follows from the K\"{o}nig-Hall Theorem.
\end{proof}

Let us note that Fact \ref{fact-4} and Corollary \ref{cor-4.2} imply
\begin{align}\label{ineq-4.1}
\sum_{1\leq i\leq 3} |\hht(B_i,[6])|\leq 6 \mbox{ for } |B_1|=|B_2|=|B_3|=2.
\end{align}
Moreover the equality is possible only if $|\hht(B_2,[6])|\geq 2$.

Since $\binom{[6]}{2}$ can be partitioned into five matchings, \eqref{ineq-4.1} implies
\begin{align}\label{ineq-4.2}
|\hht_2| \leq 5\times 6 =30.
\end{align}
Note that $|\hht_3| \leq \binom{6}{3} =20$. Hence $|\hht_1|\leq 6$ would imply $|\hht|\leq 56$.

In the sequel we tacitly assume $|\hht_1|\geq 7$. Let us solve the case when for some $x\in [6]$, $\hht(x,[6])$ is a $C_4$.  For definiteness assume
\[
\hht(1,[6]) =\{7,8\}\times \{9,10\}.
\]
In view of Fact \ref{fact-4.3} (i), $\hht(y,[6]) \subset \{(7,8),(9,10)\}$ for $2\leq y\leq 6$. If for some $y$ we have equality then again by Fact \ref{fact-4.3} (ii), $\hf(z,[6])=\emptyset$ for the remaining four values of $z\in [6]$. Hence $|\hht_1|=4+2=6$ and $|\hht|\leq 56$ follows.  Suppose $|\hht_1|\geq 7$. By symmetry and Fact \ref{fact-4.3} (i) we may assume that for some $\ell$, $4\leq \ell\leq 6$, $\hht(y,[6])=\{(7,8)\}$ for $2\leq y\leq \ell$ and $\hht(\{y\},[6])=\emptyset$ for $\ell<y\leq 6$.

Note that this readily implies
\[
|\hht_1| \leq 4+5\times 1 = 9 \mbox{ and  thereby }|\hht|\leq 59.
\]
Assuming $|\hht|\geq 57$, $|\hht_2|+|\hht_3|\geq 48$ and $|\hht_3|\geq 18$ follow.

Note that if $B_2\cup B_3=[2,6]$, $|B_2|=2$, $|B_3|=3$ and $B_2\cap B_3=\emptyset$, then Fact \ref{fact-4.1} (ii) implies that either $\hht(B_2,[6])=\emptyset$ or $B_3\notin \hht_3$. By $|\hht_3|\geq \binom{6}{3}-2$, at least eight of the ten triples in $\binom{[2,6]}{3}$ belong to $\hht_3$. Hence $\hht(B_2,[6])=\emptyset$ for at least eight choices of $B_2$. Thus,
\[
|\hht_2|\leq \left(\binom{6}{2}-8\right)\times 3=21 \mbox{ and }|\hht|\leq 9+21+20=50.
\]

This concludes the proof in the case if for some $x$,  $\hht(x,[6])$ is a $C_4$. From now on we may assume $|\hht(x,[6])|\leq 3$ for all  $x\in [6]$.

\begin{cor}\label{cor-4.5}
Let $B_1\cup B_2\cup B_3=[6]$  with $|B_i|=i$, $1\leq i\leq 3$. If   $B_3\in \hht$ then
\begin{align}\label{ineq-4.3}
|\hht(B_1,[6])|+|\hht(B_2,[6])|\leq 3.
\end{align}
\end{cor}
\begin{proof}
If $\hht(B_1,[6])$ is a $C_3$ or a path of length 3 then $\hht(B_2,[6])=\emptyset$. In the other cases $|\hht(B_1,[6])|+|\hht(B_2,[6])|\leq 3$ holds by the cross-intersecting property.
\end{proof}

Let $[6]=U_1\cup U_2$ with $U_1,U_2\in \hht$.  By  \eqref{ineq-4.3} we have
\begin{align}\label{ineq-4.4}
\sum_{y\in U_i} |\hht(y,[6])| +|\hht(U_i\setminus \{y\},[6])| \leq 3\times 3=9, \ i=1,2.
\end{align}
Since $U_1\times U_2$ can be partitioned into 3 matchings.  For each matching $P_1,P_2,P_3$, by \eqref{ineq-4.1}
\[
\sum_{1\leq i\leq 3} |\hht(P_i,[6])| \leq 6.
\]
It follows that
\begin{align}\label{ineq-4.5}
\sum_{P\in U_1\times U_2} |\hht(P,[6])| \leq 3\times 6= 18.
\end{align}
Adding \eqref{ineq-4.4} for $i=1,2$ to \eqref{ineq-4.5}, we obtain that $|\hht_1| +|\hht_2|\leq 36$. Thus,
\[
|\hht| =|\hht_1| +|\hht_2|+|\hht_3| \leq 36+\binom{6}{3} =56.
\]

\section{Proof of Lemma \ref{lem-main}}

Let $\hht\subset \binom{[n]}{3}$ be a 2-resilient family with matching number 2 and  let $u,v\in [n]$ be  arbitrary. By 2-resilience, there exist disjoint sets $T_1,T_2\in \hht(\bar{u},\bar{v})$. Let $R=T_1\cup T_2$ and $S=T_1\cup T_2\cup \{u,v\}$.

For each $P\in (u,v)\times (T_1\cup T_2)$  and $P\subset T\in \hht_{uv}$, define  $\omega(P,T)=1$ if $|T\cap S|=2$ and define  $\omega(P,T)=\frac{1}{2}$ if  $T\subset S$. Define
\[
\omega(P) =\sum_{T\colon T\in \hht_{uv}, P\subset T} \omega(P,T).
\]
 Note that an edge $(u,v,w)\in \hht$ with $w\in T_1\cup T_2$ will contribute $\frac{1}{2}$ to  both $\omega(u,w)$ and $\omega(v,w)$, and an edge $(x,y,z)\in \hht$ with $x\in \{u,v\}$ and $y,z\in T_1\cup T_2$ will contribute  $\frac{1}{2}$ to both $\omega(x,y)$ and $\omega(x,z)$. It follows that
\begin{align}\label{ineq-key8}
|\hht_{uv}| =  \sum_{P\in (u,v)\times T_1} \omega(P)+ \sum_{P\in (u,v)\times T_2} \omega(P).
\end{align}

\begin{fact}\label{fact-5.1}
For each $P\in (u,v)\times T_1$, $|\hht(P,S)|\leq 3$. Moreover, if $|\hht(P,S)|= 3$ then $P\cup \{u,v\}\notin \hht$ and there exist disjoint sets $V_1,V_2\in \hht(\overline{\{u,v\}\cup P})$ such that $|V_1\cap T_1|=|V_1\cap T_2|=1$ and $|V_2\cap T_2|= 1$.
\end{fact}

\begin{proof}
By Fact \ref{fact-3}, $\hht(x,R)$ does not contain a pseudo sunflower of size 4. It follows that $|\hht(\{u,x\},R\cup \{x\})|\leq 3$ and  $|\hht(\{u,x\},S)|= 3$ implies $(u,x,v)\notin \hht$.
 Suppose that  $(u,x,w_j)\in \hht$ with $w_j\notin S$, $j=1,2,3$.  Then by 2-resilience there exist disjoint sets $V_1,V_2\in \hht(\overline{\{u,x\}})$. Since $w_1,w_2,w_3\in (V_1\cup V_2)\setminus  (T_1\cup T_2)$,  we conclude that $|(V_1\cup V_2)\cap (T_1\cup T_2)|\leq 3$.

 If $|(V_1\cup V_2)\cap (T_1\cup T_2)|= 2$, then by symmetry we may assume $|V_i\cap T_i|=1$, $i=1,2$ and $w_3\in V_2\setminus T_2$. It follows that $V_1,T_2,P\cup\{w_3\}$ form a matching, contradicting $\nu(\hht)=2$.  Thus, $|(V_1\cup V_2)\cap (T_1\cup T_2)|= 3$ and $u,v\notin V_1\cup V_2$.
If one of $V_1,V_2$, say $V_1$, is disjoint to $T_2$, then one can find $w\in V_2\setminus R$ such that $\{u,x,w\}, V_1, T_2$ form a matching in $\hht$, a contradiction. Thus  by symmetry we may assume that  $|V_1\cap T_1|=|V_1\cap T_2|=1$ and $|V_2\cap T_2|= 1$.
\end{proof}

For $P \in (u,v)\times T_i$, let
\[
\hht_2(P)=  \left\{ w\in S\colon P\cup \{w\}\in \hht \right\}.
\]

\begin{fact}\label{fact-5.4}
For any $P \in (u,v)\times T_i$,
\begin{align}\label{ineq-5.4}
\omega(P) \leq |\hht(P,S)| +\frac{|\hht_2(P)|}{2}\leq |\hht(P,S)| +\frac{6-|\hht(P,S)|}{2}.
\end{align}
\end{fact}
\begin{proof}
By Corollary  \ref{cor-1} we have $|\hht(P)|\leq 6$. Then
\[
|\hht_2(P)|=|\hht(P)|- |\hht(P,S)|\leq 6- |\hht(P,S)|
\]
and the fact follows.
\end{proof}

Let us first prove a weaker version of Lemma \ref{lem-main}.

\begin{lem}\label{lem-5.1}
Let $\hht\subset \binom{[n]}{3}$ be a 2-resilient family with matching number 2. Then for any $u,v\in [n]$,
\[
|\hht_{uv}|\leq 42.
\]
\end{lem}

\begin{proof}
Let $T_1,T_2\in \hht(\bar{u},\bar{v})$ be disjoint sets.
Then $(u,v)\times T_i$ can be partitioned into 3 groups with  each forming a  2-matching $(u,x),(v,y)$, $i=1,2$.

Since $\hht(\{u,x\}, S),\hht(\{v,y\}, S)$ are cross-intersecting (noting that $x,y\in T_i$), we infer that either one of them is empty or they  are an identical 1-element set. Let $z=T_i\setminus \{x,y\}$.

{\bf Case A. } One of them is empty, say $\hht(\{v,y\}, S)=\emptyset$.

Then $(v,y,z)\notin \hht$. By Corollary \ref{cor-1},  $|\hht(P)|\leq 6$ for each $P\in \binom{[n]}{2}$.  By Fact \ref{fact-3} we infer that
 \[
 \omega(u,x)\leq 3+\frac{3}{2}.
 \]
 Moreover $\hht(\{v,y\})\subset S\setminus \{v,y,z\}$ implies $\omega(v,y)\leq \frac{5}{2}$.
Thus,
\begin{align*}
\omega(u,x)+\omega(v,y)\leq (3+\frac{3}{2})+\frac{5}{2}=7.
\end{align*}

{\bf Case B. } $\hht(\{u,x\}, S)=\hht(\{v,y\}, S)=\{\{t\}\}$ for some $t\in [n]\setminus S$.

Then
\begin{align*}
\omega(u,x)+\omega(v,y)\leq 2\times (1+\frac{5}{2})= 7.
\end{align*}
Thus,
\[
\sum_{P\in (u,v)\times T_i} \omega(P) \leq 3\times 7=21.
\]
Therefore,
\[
|\hht_{uv}| \leq \sum_{P\in (u,v)\times (T_1\cup T_2)} \omega(P)\leq 21+21=42.
\]
\end{proof}

\begin{fact}\label{fact-5.2}
Let $P\in (u,v)\times T_1$. If $|\hht(P,S)|\geq 1$ then $T_1\cup \{u,v\}\setminus P \notin \hht$.
\end{fact}
\begin{proof}
Indeed, assume that $P=(u,x)$ and $T_1\cup \{u,v\}\setminus P=(v,y,z) \in \hht$.  Then $(u,x,w), (v,y,z), T_2$ form a matching for some $w\notin S$, contradicting $\nu(\hht)\leq 2$.
\end{proof}

For $P\in (u,v)\times T_i$, let
\[
\omega_{in}(P) =\omega(P)-|\hht(P,S)|.
\]

\begin{fact}\label{fact-5.3}
Let $T_1=(x,y,z)$. If $|\hht(\{u,x\},S)|=3$, then
\begin{align}\label{ineq-5.1}
\omega_{in}(v,x)+\omega(v,y)+\omega(v,z)\leq 3.
\end{align}
\end{fact}
\begin{proof}
By Fact  \ref{fact-5.1} there exist disjoint sets $V_1,V_2\in \hht(\overline{\{u,x\}})$ such that $|V_1\cap T_1|=|V_1\cap T_2|=1$ and $|V_2\cap T_2|= 1$. By symmetry assume that $V_1\cap T_1=\{z\}$, $V_1\cap T_2=\{z'\}$, $V_2\cap T_2=\{y'\}$. Since $\hht(\{u,x\},S), \hht(\{v,y\},S)$ are cross-intersecting, we infer that $\hht(\{v,y\},S)=\emptyset$. Similarly $\hht(\{v,z\},S)=\emptyset$.
By Fact \ref{fact-5.2}, $\{v,y,z\}\notin \hht$. It follows that $\hht(\{v,y\})\subset \{y',z'\}$. However, $\{v,y,y'\}$ is disjoint to $V_1\cup \{u,x,w\}$ for some $w\in V_2\setminus T_2$ and $\{v,y,z'\}$ is disjoint to $V_2\cup \{u,x,w'\}$ for some $w'\in V_1\setminus R$. Thus $\hht(\{v,y\})=\emptyset$.

By the same argument, one can show that $\hht(\{v,z\})\subset \{u,x,y'\}$. Moreover, since $(v,x)$ is disjoint to $V_1\cup V_2$, we infer that $\hht(v,x) \cap S\subset \{z,y',z'\}$. Thus,
 \[
 \omega_{in}(v,x)+\omega(v,y)+\omega(v,z)\leq \frac{3}{2}+0+\frac{3}{2}=3.
 \]
\end{proof}

\begin{lem}\label{lem-5.2}
If there exists $P_0\in (u,v)\times T_1$ such that $|\hht(P_0,S)|=3$, then
\begin{align}\label{ineq-5.2}
\sum_{P\in (u,v)\times T_1} \omega(P) \leq 18.
\end{align}
\end{lem}

\begin{proof}
By symmetry assume that $T_1=(x,y,z)$ and $P_0=(u,x)$. We distinguish two cases.

{\bf Case A.} One of $|\hht(\{u,y\},S)|$, $|\hht(\{u,z\},S)|$ is greater than 2.

Since $\hht(\{v,x\},S), \hht(\{u,w\},S)$ are cross-intersecting for $w=y$ or $z$,  $\hht(\{v,x\},S)=\emptyset$. By \eqref{ineq-5.4} and \eqref{ineq-5.1},
\[
\sum_{P\in (u,v)\times T_i} \omega(P) \leq (3+\frac{3}{2})\times 3+ 3 \leq 16+\frac{1}{2}.
\]

{\bf Case B.} Both  $|\hht(\{u,y\},S)|$ and $|\hht(\{u,z\},S)|$ are smaller than 1.

 By \eqref{ineq-5.4}, $\omega(u,y)\leq 1+\frac{5}{2}$ and $\omega(u,z)\leq 1+\frac{5}{2}$. By \eqref{ineq-5.1} and $|\hht(\{v,x\},S)|\leq 3$,
\[
\sum_{P\in (u,v)\times T_i} \omega(P) \leq (3+\frac{3}{2})+(1+\frac{5}{2})\times 2+|\hht(\{v,x\},S)|+ 3 \leq 17+\frac{1}{2}.
\]
\end{proof}

Let
\[
\hg_2=\left\{P\in (u,v)\times T_1\colon |\hht(P,S)| = 2\right\} \mbox{ and }  \hg_1=\left\{P\in (u,v)\times T_1\colon |\hht(P,S)| \geq 1\right\}.
\]

The following fact is evident by $\nu(\hht)=2$.

\begin{fact}\label{fact-5.5}
$\hg_2\subset \hg_1$, $\hg_2$ is intersecting and $\hg_1$, $\hg_2$ are cross-intersecting.
\end{fact}

\begin{lem}\label{lem-5.3}
If $\hg_2=\emptyset$, then
\begin{align}\label{ineq-5.3}
\sum_{P\in (u,v)\times T_1} \omega(P) \leq 18.
\end{align}
\end{lem}

\begin{proof}
By Lemma \ref{lem-5.2}, we may assume that $|\hht(P,S)|\leq 1$ for all $P\in (u,v)\times T_1$. By Fact \ref{fact-5.2}, we infer that
\begin{align*}
\sum_{P\in (u,v)\times T_1} \omega_{in}(P) &= 6\times  \frac{|T_2|+|\{u,v\}\setminus P|}{2}+|\{T\in \hht\colon |T\cap \{u,v\}|=1,\ |T\cap T_1|=2\}|\\[2pt]
&\leq  6\times  \frac{4}{2}+6-|\hg_1|\\[2pt]
&= 18-|\hg_1|.
\end{align*}
Thus,
\[
\sum_{P\in (u,v)\times T_1} \omega(P) \leq  |\hg_1|+18-|\hg_1|=18.
\]
\end{proof}

\begin{lem}\label{lem-3.2}
For $i=1,2$,
\begin{align}\label{ineq-3.7}
\sum_{P\in (u,v)\times T_i} \omega(P) \leq 18.
\end{align}
\end{lem}

\begin{proof}
By symmetry it is sufficient to prove \eqref{ineq-3.7} for $i=1$. By Fact \ref{fact-5.5}, $\hg_2$ is a star. By Lemmas \ref{lem-5.1} and \ref{lem-5.2}, we may assume that $|\hht(P,S)|\leq 2$ for all $P\in (u,v)\times T_1$ and $\hg_2\neq \emptyset$.
 Let $T_1=\{x,y,z\}$ and  distinguish three cases.

\vspace{3pt}
{\bf Case 1. } $|\hg_2|=3$.
\vspace{3pt}

Without loss of generality assume $\hg_2=\{(u,x),(u,y),(u,z)\}$. Using Fact \ref{fact-5.5} we infer that $\hg_1=\hg_2$, that is,
\[
\hht(\{v,x\},S) = \hht(\{v,y\},S)=\hht(\{v,z\},S)=\emptyset.
\]
Moreover, by Fact \ref{fact-5.2}, $(v,x,y),(v,x,z), (v,y,z)\notin \hht$. Thus,
\[
\sum_{P\in (u,v)\times T_1} \omega(P) \leq  (2+\frac{4}{2})\times 3+ \frac{4}{2}\times 3 = 18.
\]

\vspace{3pt}
{\bf Case 2. } $|\hg_2|=2$.
\vspace{3pt}

\vspace{3pt}
{\bf Subcase 2.1.} $\hg_2=\{(u,x),(u,y)\}$.
\vspace{3pt}

Since $\hg_1,\hg_2$ are cross-intersecting, we infer that $\hg_1\subset \{(5,x),(5,y),(5,z)\}$. Thus,
\[
\hht(\{v,x\},S) = \hht(\{v,y\},S)=\hht(\{v,z\},S)=\emptyset.
\]
Then by Fact \ref{fact-5.2} $(v,x,z)$, $(v,y,z)\notin \hht$. If $|\hht(\{u,z\},S)|=1$, then by Fact \ref{fact-5.2} we also have $(v,x,y)\notin \hht$. It follows that
\[
\sum_{P\in (u,v)\times T_1} \omega(P) \leq (2+\frac{4}{2})\times 2+ (1+\frac{5}{2})+\frac{4}{2}\times 3 = 17+\frac{1}{2}.
\]
If $\hht(\{u,z\},S)=\emptyset$, then
\[
\sum_{P\in (u,v)\times T_1} \omega(P) \leq  (2+\frac{4}{2})\times 2+ \frac{6}{2}+ \frac{5}{2}\times 2+\frac{4}{2} = 18.
\]

\vspace{3pt}
{\bf Subcase 2.2.} $\hg_2=\{(u,x),(v,x)\}$.
\vspace{3pt}

Since $\hg_1,\hg_2$ are cross-intersecting, we infer that $\hg_1=\hg_2$. It follows that
\[
\hht(\{u,y\},S) = \hht(\{u,z\},S)=\hht(\{v,y\},S)=\hht(\{v,z\},S)=\emptyset.
\]
By Fact \ref{fact-5.2} we also have $(u,y,z),(v,y,z)\notin \hht$. Thus,
\[
\sum_{P\in (u,v)\times T_1} \omega(P) \leq  (2+\frac{4}{2})\times 2+ \frac{5}{2}\times 4 = 18.
\]

\vspace{3pt}
{\bf Case 3. } $|\hg_2|=1$.
\vspace{3pt}

By symmetry we may assume $\hg_2=\{(u,x)\}$. Then by $\nu(\hht)\leq 2$ we infer that $|\hg_1|\leq 4$ and $\hht(\{v,y\},S)=\hht(\{v,z\},S)=\emptyset$.

If $\{(u,x),(u,y),(u,z)\}\subset \hg_1$, then $(v,x,y),(v,y,z),(v,x,z)\notin \hht$. It follows that
\[
\sum_{P\in (u,v)\times T_1} \omega(P) \leq  (2+\frac{4}{2})+(1+\frac{5}{2})\times 2+1+ \frac{4}{2}\times 3 = 18.
\]
If $\hg_1=\{(u,x),(u,y),(v,x)\}$, then $(v,y,z),(v,x,z),(u,y,z)\notin \hht$. Note that $(u,y,z)\notin \hht$ implies $w(u,z)\leq \frac{5}{2}$. It follows that
\[
\sum_{P\in (u,v)\times T_1} \omega(P) \leq  (2+\frac{4}{2})+(1+\frac{5}{2})+\frac{5}{2}+(1+\frac{5}{2})+\frac{5}{2}+\frac{4}{2}= 18.
\]
If $\hg_1=\{(u,x),(u,y)\}$, then $(v,y,z),(v,x,z)\notin \hht$. It follows that
\[
\sum_{P\in (u,v)\times T_1} \omega(P) \leq  (2+\frac{4}{2})+(1+\frac{5}{2})+\frac{6}{2}+ \frac{5}{2}\times 2 +\frac{4}{2}= 17+\frac{1}{2}.
\]
If $\hg_1=\{(u,x),(v,x)\}$, then $(u,y,z),(v,y,z)\notin \hht$. It follows that
\[
\sum_{P\in (u,v)\times T_1} \omega(P) \leq   (2+\frac{4}{2})+(1+\frac{5}{2})+ \frac{5}{2}\times 4= 17+\frac{1}{2}.
\]
If $\hg_1=\{(u,x)\}$, then $(v,y,z)\notin \hht$. It follows that
\[
\sum_{P\in (u,v)\times T_1} \omega(P) \leq  (2+\frac{4}{2})+ \frac{6}{2}\times 3+\frac{5}{2}\times 2 = 18.
\]
\end{proof}

\begin{proof}[Proof of Lemma \ref{lem-main}]
By Lemma \ref{lem-3.2} and \eqref{ineq-key8} we conclude that
\[
|\hht_{uv}| =\sum_{P\in (u,v)\times T_1} \omega(P)+\sum_{P\in (u,v)\times T_2} \omega(P) \leq 18\times 2 =36.
\]
\end{proof}

\begin{lem}\label{lem-3.3}
 Suppose that $\binom{[4]}{3}\subset \hht$ and $(6,7,8),(6,9,10),(7,9,w)\in \hht$ with some $w\notin [6]$. Then
 \[
|\hht_{56}|\leq 28.
 \]
\end{lem}

\begin{proof}
Let $T_1=(1,2,3)$, $T_2=(7,9,w)$ and $S=T_1\cup T_2\cup \{5,6\}$.
By $\nu(\hht)\leq 2$ and $\binom{[4]}{3}\subset \hht$, we infer that $\hht(P, S)\subset \{\{4\}\}$ for all $P\in (5,6)\times T_1$ and $(5,6,i)\notin \hht$ for all $i\in T_1$. Thus, $\{P\in(5,6)\times T_1\colon |\hht(P,S)|\geq 2\}=\emptyset$.

 Note that $(6,7,8),(6,9,10)\in \hht$. If $(5,i,j)\in \hht$ with $i\in T_1$, $j\in T_2$ then either $(5,i,j), [4]\setminus \{i\}, (6,7,8)$ or $(5,i,j), [4]\setminus \{i\}, (6,9,10)$ form a matching of size 3, contradicting $\nu(\hht)\leq 2$. Thus $(5,i,j)\notin \hht$ for all $i\in T_1$, $j\in T_2$. Therefore, $\hht(\{5,i\},S)\subset T_1\setminus \{i\}$.

Recall that
\[
\hg_1=\{P\in(5,6)\times T_1\colon |\hht(P,S)|= 1\}.
\]
By Fact \ref{fact-5.2},
\[
\sum_{P\in (5,6)\times T_1} \omega(P) =|\hg_1|+\sum_{P\in (5,6)\times T_1} \omega_{in}(P) \leq |\hg_1| +(6-|\hg_1|) + \frac{3}{2}\times 3 = 10+\frac{1}{2}.
\]
By Lemma \ref{lem-3.2}, we also have
\[
\sum_{P\in (5,6)\times T_2} \omega(P) \leq 18.
\]
Thus we conclude that
\[
|\hht_{56}|= \sum_{P\in (5,6)\times T_1} \omega(P)+\sum_{P\in (5,6)\times T_2} \omega(P) \leq 28+\frac{1}{2}.
\]
\end{proof}

\begin{lem}\label{lem-3.4}
 Suppose that $|\hht\cap\binom{[4]}{3}|\geq 3$ and $(6,7,8),(6,9,10),(7,9,w)\in \hht$ with some $w\notin [6]$. Then
 \[
|\hht_{56}|\leq 34.
 \]
\end{lem}

\begin{proof}
Without loss of generality assume $(1,2,3),(1,2,4),(1,3,4)\in \hht$. Let $T_1=(1,2,3)$, $T_2=(7,9,w)$ and $S=T_1\cup T_2\cup \{5,6\}$.
By $\nu(\hht)\leq 2$, we infer that $\hht(P, T_1)\subset \{\{4\}\}$ for all $P\in (5,6)\times (2,3)$ and $(5,6,i)\notin \hht$ for $i=2,3$. Let
\[
\hg_2=\left\{P\in (u,v)\times T_1\colon |\hht(P,S)| \geq 2\right\}.
\]
Then $\hg_2\subset \{(5,1),(6,1)\}$.

Recall that $(6,7,8),(6,9,10)\in \hht$. If $(5,i,j)\in \hht$ for $i\in \{2,3\}$ and $j\in T_2$, then either $(5,i,j), [4]\setminus \{i\}, (6,7,8)$ or $(5,i,j), [4]\setminus \{i\}, (6,9,10)$ form a matching of size 3, contradicting $\nu(\hht)\leq 2$. Thus $(5,i,j)\notin \hht$ for all $i\in (2,3)$, $j\in T_2$ and $\hht(\{5,2\},S)\subset [4]$, $\hht(\{5,3\},S)\subset  [4]$ follow. Therefore,
\begin{align}\label{ineq-5.5}
\sum_{P\in (5,6)\times T_1} \omega_{in}(P) \leq  (6-|\hg_1|)+ \sum_{P\in (5,6)\times T_1\setminus \{(5,2),(5,3)\}} |\hht(P)\cap (T_2\cup \{5,6\})|.
\end{align}

Note that  if $P\in \hg_1$ then by Fact \ref{fact-5.2}, $(T_1\cup \{5,6\})\setminus P\notin \hht$. If $\hg_2=\emptyset$, then by \eqref{ineq-5.5}
\[
\sum_{P\in (5,6)\times T_1} \omega(P) \leq |\hg_1| +(6-|\hg_1|) + \frac{4}{2}\times 4 = 14<16.
\]
If $|\hg_2|=1$, then by \eqref{ineq-5.5}
\[
\sum_{P\in (5,6)\times T_1} \omega(P) \leq 3+(|\hg_1|-1)+(6-|\hg_1|)+ \frac{4}{2}\times 4 = 16.
\]
Thus we may assume $\hg_2=\{(5,1),(6,1)\}=\hg_1$.

Then $(5,2,3),(6,2,3)\notin \hht$ and
\[
\hht(\{5,2\},S)=\hht(\{5,3\},S)=\hht(\{6,2\},S)=\hht(\{6,3\},S)=\emptyset.
\]
If one of $|\hht(\{5,1\},S)|,|\hht(\{6,1\},S)|$ equals 3, then by Fact \ref{fact-5.1}, $(5,6,1)\notin \hht$. Recall that $|\hg_1|=2$. By \eqref{ineq-5.5}, it follows that
\[
\sum_{P\in (5,6)\times T_1} \omega(P) \leq 3\times 2 +(6-|\hg_1|) + \frac{3}{2}\times 4 = 16.
\]
If both of $|\hht(\{5,1\},S)|,|\hht(\{6,1\},S)|$ equal 2, then by \eqref{ineq-5.5}
\[
\sum_{P\in (5,6)\times T_1} \omega(P) \leq 2\times 2 +(6-|\hg_1|) + \frac{4}{2}\times 4 = 16.
\]
Thus, by Lemma \ref{lem-3.2} we have
\[
|\hht_{56}|= \sum_{P\in (5,6)\times T_1} \omega(P)+\sum_{P\in (5,6)\times T_2} \omega(P)\leq 16+18 = 34.
\]
\end{proof}

\section{The case $\hr$ is not 3-intersecting}

In this section, we assume that $\hr=\hht^2$ is not 3-intersecting. Then one can choose $R_1,R_2\in \hr$ with $|R_1\cap R_2|=2$ and  using $\nu(\hht)=2$ it is easy to see that up to isomorphism the four edges of $\hht$ involved are  $(1,2,5),(3,4,5),(6,7,8),(6,9,10)$.

Define the partition $\hht=\hht_0\cup \hht_1$ by
\[
\hht_0=\left\{T\in \hht\colon (5,6)\cap T=\emptyset\right\} \mbox{ and }\hht_1=\left\{T\in \hht\colon |(5,6)\cap T|\geq 1\right\}.
\]
Clearly $|\hht_1|=|\hht_{56}|$.

\begin{fact}\label{fact-3.1}
For $P\in (1,2)\times (3,4)$, $|\hht_0(P,[4])|\leq 5$. Moreover, if $|\hht_0(P,[4])|= 5$ then $\hht\cap \binom{[4]}{3}=\emptyset$. Similarly, for $P\in (7,8)\times (9,10)$, $|\hht_0(P,[7,10])|\leq 5$ and if $|\hht_0(P,[7,10])|= 5$ then $\hht\cap \binom{[7,10]}{3}=\emptyset$.
\end{fact}

\begin{proof}
Let us assume $P=(1,3)$. We claim that $|\hht_0(\{1,3\},[4])|\leq 5$. Indeed, if $(1,3,x_i)\in \hht_0$ with $x_i\notin [6]$, $i=1,2,\ldots,6$ then $(1,3,x_1),\ldots,(1,3,x_6),(1,2,5)$ form a pseudo sunflower of size 7, contradicting Fact \ref{fact-2}. Thus $|\hht_0(\{1,3\},[4])|\leq 5$.

If $(1,3,x_i)\in \hht_0$ with $x_i\notin [6]$, $i=1,2,\ldots,5$, then by 2-resilience there exists $R\in \hr(\bar{1},\bar{3})$ with $x_1,x_2,\ldots,x_5\in R$. Since $\hht,\hr$ are cross-intersecting, by $(1,2,5)\cap R\neq \emptyset\neq (3,4,5)\cap R$ we infer that $R=\{5,x_1,x_2,\ldots,x_5\}$. Then $\hht\cap \binom{[4]}{3}=\emptyset$ follows from $\nu(\hht)\leq 2$.
\end{proof}

\begin{lem}\label{lem-3.1}
$|\hht_0|\leq 24$. Moreover, if $|\hht_0|\geq  23$ then by symmetry we may assume that $\binom{[4]}{3}\subset \hht$ and $(7,9,w)\in \hht$ for some $w\notin [6]$. If $|\hht_0|\geq  21$ then by symmetry we may assume that  $|\hht\cap \binom{[4]}{3}|\geq 3$ and $(7,9,w)\in \hht$ for some $w\notin [6]$.
\end{lem}

\begin{proof}
Let $T\in \hht_0$.
By $\nu(\hht)\leq 2$, we infer that either both $T\cap  (1,2)$ and $T\cap (3,4)$ are non-empty or both $T\cap  (7,8)$ and $T\cap (9,10)$ are non-empty.
For each $P\in (1,2)\times (3,4)$ and $P\subset T\in \hht_0$, define $\omega(P,T)=1$ if $|T\cap[4]|=2$ and  $\omega(P,T)=\frac{1}{2}$ if $T\subset [4]$. Similarly, for each $P\in (7,8)\times (9,10)$ and $P\subset T\in \hht_0$, define $\omega(P,T)=1$ if $|T\cap[7,10]|=2$ and  $\omega(P,T)=\frac{1}{2}$ if $T\subset [7,10]$. Define
\[
\omega(P) =\sum_{P\subset T\in \hht_0} \omega(P,T).
\]
Then it is easy to check that
\begin{align}\label{ineq-2.1}
|\hht_0| = \sum_{P\in (1,2)\times (3,4)} \omega(P)+ \sum_{P\in (7,8)\times (9,10)} \omega(P).
\end{align}

Let us prove that
\begin{align}\label{ineq-new2.1}
 \omega(P)+ \omega(P')\leq 6
\end{align}
whenever $P,P'$ are disjoint pairs in $(1,2)\times (3,4)\cup (7,8)\times (9,10)$ with $P\cup P'=[4]$ or $P\cup P'=[7,10]$. By symmetry let $P=(1,3)$, $P'=(2,4)$.

Assume first that both $\hht_0(P,[4])$ and $\hht_0(P',[4])$ are non-empty.
If there exist  $x\neq y$ with $x\in \hht_0(P)$, $y\in \hht_0(P')$, then by $\nu(\hht)\leq 2$ both $(x,y)\cap (7,8)$ and $(x,y)\cap (9,10)$ are non-empty. In this case we infer
\begin{align*}
|\hht_0(\{1,3\},[4])|+|\hht_0(\{2,4\},[4])| \leq 4.
\end{align*}
Therefore,
\begin{align}\label{ineq-3.4}
\omega(1,3)+\omega(2,4) \leq 4+\frac{1}{2}\times 4 =6,
\end{align}
with equality holding only if $\binom{[4]}{3}\subset \hht_0$.

Assume next that one  of $\hht_0(\{1,3\},[4]), \hht_0(\{2,4\},[4])$ is empty, say $\hht_0(\{2,4\},[4])=\emptyset$. By Fact \ref{fact-3.1} $|\hht_0(\{1,3\},[4])|\leq 5$. If $|\hht_0(\{1,3\},[4])|\leq 4$, then
\begin{align}\label{ineq-3.5}
\omega(1,3)+\omega(2,4) \leq (4+\frac{2}{2})+ (\frac{1}{2}+\frac{1}{2}) =6,
\end{align}
with equality holding only if $\binom{[4]}{3}\subset \hht_0$.
If $|\hht_0(\{1,3\},[4])|= 5$ then by Fact \ref{fact-3.1} $\hht\cap \binom{[4]}{3}=\emptyset$. It follows that
\begin{align}\label{ineq-3.6}
\omega(1,3)+\omega(2,4) \leq 5<6.
\end{align}
This proves \eqref{ineq-new2.1}.
Thus,
\[
|\hht_0| = \sum_{P\in (1,2)\times (3,4)\cup (7,8)\times (9,10)} \omega(P) \leq 4\times 6 = 24.
\]

If $|\hht_0|\geq 23$, then equality holds for at least one of the four pairs $(P,P')$. By symmetry assume that equality holds for the pair $((1,3),(2,4))$ and then $\binom{[4]}{3}\subset \hht$ follows from \eqref{ineq-3.4} and \eqref{ineq-3.5}.  Moreover, $|\hht_0|\geq 23$ also implies that  one of $\omega(7,9)+\omega(8,10) \geq 5+\frac{1}{2}$ and $\omega(7,10)+\omega(8,9) \geq 5+\frac{1}{2}$  holds. By symmetry assume that $\omega(7,9)+\omega(8,10) \geq 5+\frac{1}{2}$. Then $|\hht\cap \binom{[7,10]}{3}|\geq 3$. Thus we may assume $(7,9,8)\in \hht$.

Suppose that  $|\hht_0|\geq 21$. If $\omega(P)+\omega(P')=6$ for one of the four pairs $(P,P')$,  by symmetry assume  $\omega(1,3)+\omega(2,4)=6$, then $\binom{[4]}{3}\subset \hht$ follows. Moreover, $|\hht_0|\geq 21$ also implies that  one of $\omega(7,9)+\omega(8,10) \geq 4+\frac{1}{2}$ and $\omega(7,10)+\omega(8,9) \geq 4+\frac{1}{2}$  holds. By symmetry assume that $\omega(7,9)+\omega(8,10) \geq 4+\frac{1}{2}$. Then either $|\hht\cap \binom{[7,10]}{3}|\geq 1$ or one of $|\hht_0(\{7,9\},[7,10])|,|\hht_0(\{8,10\},[7,10])|$ equals 5. In either case, we may assume $(7,9,w)\in \hht$ for some $w\notin [6]$.

If $\omega(P)+\omega(P')<6$ holds for all of the four pairs $(P,P')$, then by $|\hht_0|\geq 21$ there exists  $(P,P')$ such that $\omega(P)+\omega(P')=5+\frac{1}{2}$. By symmetry assume  $\omega(1,3)+\omega(2,4) \geq 5+\frac{1}{2}$. Then  $|\hht \cap \binom{[4]}{3}|\geq 3$ follows from \eqref{ineq-3.4} and \eqref{ineq-3.5}.  Moreover, by symmetry we may also assume $\omega(7,9)+\omega(8,10) \geq 5$. Then by the same argument we may assume $(7,9,w)\in \hht$ for some $w\notin [6]$. Thus the lemma is proven.
\end{proof}

%

Now we are in a position to conclude the proof of Theorem \ref{thm-main}.

\begin{proof}[Proof of Theorem \ref{thm-main}]
Let $\hht\subset \binom{[n]}{3}$ be a 2-resilient family with $\nu(\hht)=2$. By the proof in Section 4, we may assume that $\hr=\hht^2$ is not 3-intersecting. Then we may assume  $(1,2,5),(3,4,5),(6,7,8),(6,9,10)\in \hht$.

By Lemma \ref{lem-main}, we infer that $|\hht_1|=|\hht_{56}|  \leq 36$.
If $|\hht_0|\leq 20$ then
\[
|\hht| =|\hht_0|+|\hht_1| \leq 20+36=56
\]
and we are done. Thus we may assume that $|\hht_0|\geq 21$.

If $|\hht_0|\geq 23$, then by Lemma \ref{lem-3.1} we may assume $\binom{[4]}{3}\subset \hht$ and $(7,9,w)\in \hht$ for some $w\notin [6]$. Then by Lemmas \ref{lem-3.1} and \ref{lem-3.3},
\[
|\hht| =|\hht_0|+|\hht_1| \leq 24+28<56.
\]

Finally assume $21\leq |\hht_0|\leq  22$. By Lemma \ref{lem-3.1},  we may assume $|\hht\cap \binom{[4]}{3}|\geq 3$ and $(7,9,w)\in \hht$ for some $w\notin [6]$. Then by Lemma \ref{lem-3.4} we conclude that
\[
|\hht| =|\hht_0|+|\hht_1| \leq 22+34=56.
\]
\end{proof}

\section{Concluding remarks}

In the present paper we were dealing with 3-graphs in the context of Conjecture \ref{conj-1}. The ratio of the bound in Theorem \ref{thm-main3} and that of conjecture is about $\frac{73}{27}\sim 2.7$, which is still very large. The case of $(k-1)$-resilient $k$-graphs seems to be much harder. The natural construction $\binom{[ks+k-1]}{k}$ is no longer optimal for small values of $s$. However for $k$ fixed and $s\rightarrow \infty$ we could not find any better constructions. This motivates the following rather audacious conjecture.

\begin{conj}\label{conj-6.1}
For fixed $k\geq 4$ and $s\geq s_0(k)$,
\begin{align}
m(k,s)=\binom{ks+k-1}{k}.
\end{align}
\end{conj}

Let us make a weaker version of Conjecture \ref{conj-6.1} too.

 \begin{conj}\label{conj-6.2}
Suppose that the $k$-graph $\hf$ satisfies, $\nu (\hf) =s$  and
 $\tau(\hf)=sk$. Then
\[
|\hf| \leq {sk+k-1 \choose k} \mbox{ for } s>s_0(k).
\]
\end{conj}

Finally let us recall an analogous problem for $t$-intersecting $k$-graphs with $t$-covering number $k$.

\begin{thm}[\cite{F25}]
Let $k>t\geq 1$ be integers and $\hf$ a $t$-intersecting $k$-graph with the property that to any $(k-1)$-element set $E$ there exists
some $F\in \hf$ with $|F\cap E|<t$. If $k\geq (k-t)^4$ then
\begin{align}\label{ineq-6.2}
|\hf| \leq \binom{2k-t}{k}
\end{align}
and equality holds only if $\hf=\binom{X}{k}$ for some $(2k-t)$-set $X$.
\end{thm}

It is conjectured in \cite{F25} that \eqref{ineq-6.2} holds for $k>c(k-t)$ for some absolute constant $c$.

\end{document}